\documentclass[10 pt]{amsart}
\usepackage{amsmath,amssymb,amsthm,amscd}
\usepackage{graphicx}
\usepackage[margin=1.5in]{geometry}
\usepackage{tikz}
\usepackage[OT2,T1]{fontenc}
\usepackage{comment}
\usepackage{tabulary}

\newtheorem{thm}{Theorem}[section]
\newtheorem{cor}[thm]{Corollary}
\newtheorem{lem}[thm]{Lemma}
\newtheorem{pro}[thm]{Proposition}
\theoremstyle{definition}

\newtheorem{rem}[thm]{Remark}
\newtheorem{exa}[thm]{Example}

\newcommand{\Gal}{\mathrm{Gal}}

\newcommand{\Hom}[0]{\mathrm{Hom}}
\newcommand{\End}[0]{\mathrm{End}}

\newcommand{\ns}[0]{\mathrm{ns}}

\newcommand{\Sel}[0]{\mathrm{Sel}}
\newcommand{\rk}[0]{\mathrm{rk}}
\newcommand{\coker}[0]{\mathrm{coker\hspace{.5mm}}}
\newcommand{\Tr}[0]{\mathrm{Tr}}
\newcommand{\ord}[0]{\mathrm{ord}}
\newcommand{\res}[0]{\mathrm{res}}
\newcommand{\Res}[3]{\mathrm{Res}^{#1}_{#2}#3}
\newcommand{\cores}[0]{\mathrm{cor}}

\DeclareSymbolFont{cyrletters}{OT2}{wncyr}{m}{n}
\DeclareMathSymbol{\Sha}{\mathalpha}{cyrletters}{"58}

\begin{document}

\pagestyle{plain}
\setcounter{page}{1}

\title{Selmer groups of elliptic curves in degree $p$ extensions}
\author{Julio Brau}

\begin{abstract}
\noindent We study the growth of the Galois invariants of the $p$-Selmer group of an elliptic curve in a degree $p$ Galois extension. We show that this growth is determined by certain local cohomology groups and determine necessary and sufficient conditions for these groups to be trivial. Under certain hypotheses this allows us to give necessary and sufficient conditions for there to be growth in the full $p$-Selmer group in a degree $p$ Galois extension.
\end{abstract}

\maketitle

\section{Introduction}

Let $K$ be a number field and $E/K$ an elliptic curve defined over $K$.  For a positive integer $n$, the $n$-Selmer group of $E$ over $K$ is defined by
$$
\Sel_n(E/K)=\ker\Big(H^1(K,E[n])\longrightarrow \prod_v H^1(K_v,E(\overline{K_v}))\Big)
$$
where $v$ runs over all places of $K$, and $E[n]$ denotes $E(\overline{K})[n]$. For $p$ an odd prime, let $L/K$ be a Galois extension of degree $p$ with Galois group $G$. Then there is a natural action of $G$ on the $p$-Selmer group of $E/L$, denoted by $\Sel_p(E/L)$. In this paper we discuss the size of $\Sel_p(E/L)^G$. Growth of the $p$-Selmer group in Galois extensions has been studied by several authors, often with the aim of obtaining unboundedness results for the Tate-Shafarevich group (see \cite{Bartel}, \cite{Bartel2}, \cite{Cesna}, \cite{Kloost}, \cite{KloostSchaef}). Indeed, the unboundedness of $\dim_{\mathbb{F}_p}\Sha(E/\mathbb{Q})[p]$ has been shown for primes $p\leqslant 7$ or $p=13$ (see \cite{CasselsSelmer}, \cite{Fisher}, \cite{Kramer} and \cite{Matsuno2}).  In \cite{MatsSha}, Matsuno showed that as $E$ varies over elliptic curves over $\mathbb{Q}$, the $\mathbb{F}_p$-dimension of the Tate-Shafarevich group, and in particular the $p$-Selmer group of $E$ over a fixed cyclic degree $p$ extension of $\mathbb{Q}$ can be arbitrarily large.

In this paper we will primarily be concerned with studying the growth of the $p$-Selmer group of an elliptic curve in a degree $p$ Galois extension, and in particular determining the exact $\mathbb{F}_p$-dimension of the Galois invariants of the $p$-Selmer group. As an example in \cite{TD}, Tim Dokchitser analyzes the growth of the $3$-Selmer group of $X_1(11)$ in cubic extensions of the form $K(\sqrt[3]{m})$, where $K=\mathbb{Q}(\zeta_3)$. He shows that the $3$-Selmer group grows if and only if $11|m$ or there exists a prime $v$ of $K$ such that $v|m$ and $\tilde{E}(k_v)[3]\neq 0$. This is done by applying a formula of Hachimori and Matsuno for the $\lambda$-invariant in $p$-power Galois extensions (Theorem 3.1 of \cite{HM} and Corollaries 3.20, 3.24 of \cite{DDIwa}). If we let $K=\mathbb{Q}(\zeta_p)$, this same method can be applied to give necessary and sufficient conditions for having growth of the $p$-Selmer group in extensions of the form $K(\sqrt[p]{m})$ for elliptic curves with ordinary reduction at primes above $p$ satisfying certain additional conditions, among them having trivial $p$-Selmer group over $K$ and trivial cyclotomic Euler characteristic.

As a corollary of our main result, we give necessary and sufficient conditions for having growth of the $p$-Selmer group in the same setting as above, however without any assumption on the Euler characteristic and without requiring ordinary reduction above $p$. Indeed, our results are mainly concerned with computing the exact $\mathbb{F}_p$-dimension of $\Sel_p(E/L)^G$ as best as we can, as opposed to questions of unboundedness. This is illustrated by the following simpler version of our main result, for which we establish the following notation.

Let $E/\mathbb{Q}$ be an elliptic curve over the rationals with $j$-invariant $j:=j(E/\mathbb{Q})$. Let $\ell$ be a prime at which $E$ has semistable reduction and $m$ a positive integer. Let $s\geqslant 0$ be such that $\ell^s \| m$, that is, such that $m=\ell^s d$ with $d$ coprime to $\ell$. Since $E$ is semistable at $\ell$ it follows that we may write $j=\ell^{-n}\frac{a}{b}$ with $n$ a positive integer and $a$ and $b$ coprime to $\ell$. Then define 
\begin{equation}\label{ellunit}
u_{\ell,m} :=d^n\bigg (\frac{a}{b} \bigg)^s.
\end{equation}

\begin{thm}\label{MainThm1}
Let $p$ be an odd prime and $E$ a semistable elliptic curve over $\mathbb{Q}$ with good reduction at $p$. Let $K=\mathbb{Q}(\zeta_p)$ and suppose that $\Sel_p(E/K)$ is trivial. Consider the family of $p$-extensions $L_m=K(\sqrt[p]{m})$. Then

$$
\dim_{\mathbb{F}_p} \Sel_p(E/L_m)^{\Gal(L_m/K)} = \sum_v \delta_v
$$
with $v$  running over all primes of $K$ and the $\delta_v$ are given by the following table: \\[3mm]

\noindent\hskip-13.5mm
\begin{tabular}{@{\vrule width 1.2pt\ }l|c|c|c@{\ \vrule width 1.2pt}}
\noalign{\hrule height 1.2pt}
Reduction type of $E$ at $v$ & $v$ ramified in $L_m/K$ & $v$ inert in $L_m/K$ & $v$ split in $L_m/K$ 
\\
\noalign{\hrule height 1.2pt}
good ordinary, $v | p$ & $
\delta_v=
\begin{cases}
1\text{ or } 2 & \text{if $\tilde{E}(k_v)[p]\neq 0$} \\
0 & \text{otherwise}
\end{cases}
$ & \multicolumn{2}{@{\ }c@{\ \vrule width 1.2pt}}{
$\delta_v = 0$
} \cr
\hline
good supersingular, $v|p$ & $
\delta_v =
\begin{cases}
p-2 & \text{if $p\mid m$} \\
0 & \text{otherwise}
\end{cases}
$ & \multicolumn{2}{@{\ }c@{\ \vrule width 1.2pt}}{
$\delta_v = 0$
} \cr
\hline
good, $v \nmid p$ & $\delta_v = \dim_{\mathbb{F}_p} 	\tilde{E}(k_v)[p]$ & \multicolumn{2}{@{\ }c@{\ \vrule width 1.2pt}}{
$\delta_v = 0$
} \cr
\hline
split multiplicative, $v \nmid p$ & $\delta_v = \begin{cases}
1 & \text{if $u_{\ell_v,m}^{\frac{q_v-1}{p}} \equiv 1 \mod \ell_v$}\\
0 & \text{otherwise}
\end{cases}
$
& $\delta_v = \begin{cases}
1 & \text{if $p \mid c_v$} \\
0 & \text{otherwise}
\end{cases}
$ & $\delta_v = 0$ \cr
\hline
nonsplit multiplicative, $v \nmid p$ & \multicolumn{3}{@{\ }c@{\ \vrule width 1.2pt}}{
$\delta_v = 0$
}\cr
\hline
\multicolumn{4}{@{\vrule width 1.2pt\ }l@{\ \vrule width 1.2pt}}{\small
Here $c_v$ is the Tamagawa number of $E/K_v$,  $\ell_v$ is the prime in $\mathbb{Q}$ below $v$ and $q_v$ is the size of the residue field of $K_v$.
}\cr
\hline
\end{tabular}

\end{thm}

\vspace{1cm}

We remark that the sum given in the theorem is finite, as there are only finitely many $v$ such that $v$ is ramified in $L_m/K$ or $p|c_v$, and $\delta_v=0$ for all other primes. Note also that the contribution of the $\delta_v$ coming from primes $v|p$ of good supersingular reduction grows arbitrarily large with $p$, and this is the only type of reduction for which this phenomenon occurs. This theorem essentially follows from the more general scenario of considering arbitrary Galois extensions of degree $p$. However in the more general setting, we are only able to obtain upper and lower bounds for the contribution of supersingular primes to the $\mathbb{F}_p$-dimension of $\Sel_p(L/K)^G$. In order to state the more general case we use the following notation. Suppose $L/K$ is a Galois extension of degree $p$ with Galois group $G$, and let $w$ be a prime of $L$. Then denote by $\pi_w$ a uniformiser for the completion $L_w$. Let $\sigma$ be a generator for $G$, and set $t_w=\ord_w(\sigma(\pi_w)-\pi_w)-1$. Recall then that $t_w$ is the largest index for which the higher ramification group of $L_w/K_v$ is non-trivial, where $v$ is a prime of $K$ below $w$. When $v$ is (totally) ramified in $L$, then $t_w$ is related to the discriminant $\Delta_{L/K}$ via the formula
$$
(t_w+1)(p-1) = \ord_v(\Delta_{L/K}).
$$
For more details see section 3.3. Now suppose $v \nmid p$ and $E$ has semistable reduction at $v$. Then we may write $L_w=K_v(\sqrt[p]{\pi_v})$ where $\pi_v$ is a uniformiser in $K_v$ (see section 5.1). Writing $\mathcal{O}_v$ for the ring of integers of $K_v$, we let $u_v$ be the unit in $\mathcal{O}_v$ such that $j(E/K_v)=\pi_v^{-m}u_v$, where $-m=\ord_v(j(E)).$

\begin{thm}\label{MainThm2}
Let $p$ be an odd prime and $E$ a semistable elliptic curve over a number field $K$ with good reduction at all primes above $p$. Let $L/K$ be a Galois extension of degree $p$. Suppose also that $\Sel_p(E/K)$ is trivial. Then
$$
\dim_{\mathbb{F}_p} \Sel_p(E/L)^{\Gal(L/K)} = \sum_v \delta_v
$$
where the $\delta_v$ are as in Theorem \ref{MainThm1}, except for the following two cases:
\begin{itemize}
\item[(i)] If $v \nmid p$ is a prime of split multiplicative reduction, we have that
$$
\delta_v=\begin{cases}
1 & \text{if $u_v^{\frac{q_v-1}{p}} \equiv 1 \mod{\pi_v}$} \\
0 & \text{otherwise}
\end{cases}.
$$
\item[(ii)] If $v|p$ is a prime of good supersingular reduction, we have that
$$
\delta_v=0 \Longleftrightarrow t_w=1
$$
where $w$ is a place of $L$ above $v$. Further, for $t_w\geqslant 2$ we have
$$
f_v \leqslant \delta_v \leqslant \min\{f_v(t_w-1), [K_v:\mathbb{Q}_p]+2\}
$$
where $f_v$ is the inertia degree of $v$ in $K/\mathbb{Q}$.
\end{itemize}

\end{thm}

\begin{cor}\label{MainCor}
Under the assumptions of Theorem 1.2, we have that $\Sel_p(E/L)$ is trivial if and only if $\delta_v=0$ for all $v$.
\end{cor}

\begin{exa}
Let $E/\mathbb{Q}$ be the elliptic curve 17a1 in Cremona's database.  It is given by the Weierstrass equation $Y^2 +XY+Y = X^3-X^2-X-14$. Denote $K=\mathbb{Q}(\zeta_3)$. This curve satisfies $E(K)\simeq \mathbb{Z}/4\mathbb{Z}$ and $\#\Sha(E/K)[3]=1$, hence $\Sel_3(E/K)$ is trivial. Let us use Theorem 1.1 and Corollary 1.3 to determine necessary and sufficient conditions for there to be change in 3-Selmer in extensions of the form $L_m = K(\sqrt[3]{m})$. First note $E$ has good supersingular reduction at $3$ and split multiplicative reduction at $17$. Since $q_v=17^2$ and $u_{17,m}$ is already defined over $\mathbb{Q}$, it follows that 
$$
u_{17,m}^{\frac{q_v-1}{3}} \equiv 1 \pmod{\pi_v},
$$
hence there is $3$-Selmer growth in every extension $K(\sqrt[3]{m})$ for all cube-free $m$ divisible by $17$. It follows from Theorem 1.1 and Corollary 1.3 that $\Sel_3(E/L_m)$ is trivial if and only if the following conditions are satisfied:
\begin{itemize}
\item[(i)] $3 \nmid m$.
\item[(ii)] $17 \nmid m$.
\item[(iii)] There does not exist a prime $v$ of $K$ such that $v|m$ and $\tilde{E}(k_v)[3]\neq 0$.
\end{itemize}
Note that there are infinitely many cube-free $m$ satisfying the above three conditions and hence there are infinitely many extensions of the form $L_m$ such that $\rk\ E/L_m = 0$.  We remark that there are also infinitely many $v$ such that  $\tilde{E}(k_v)[3]\neq 0$. This can be seen for example by considering primes which split completely in $K(E[3])$. Incidentally this shows that $\dim_{\mathbb{F}_3}\Sel_3(E/L_m)$ is unbounded as $m$ varies over cube-free integers. The first few primes $v$ for which $\tilde{E}(k_v)[3]\neq 0$ are the primes above $11, 19$ and $29$.
\end{exa}

\begin{rem}
Note that the conditions obtained in the previous example are compatible with what is predicted by the global root number. Indeed, for $m$ divisible by $17$ parity predicts that $\rk\ E/\mathbb{Q}(\sqrt[3]{m})$ is odd and hence positive. It follows that the rank and hence $3$-Selmer should grow in every extension of the form $K(\sqrt[3]{m})$ for cube-free $m$ divisible by $17$, which is what we had previously obtained. 
\end{rem}

The idea of the proofs is to relate the $\mathbb{F}_p$-dimension of the Galois invariants of Selmer to the sum of the $\mathbb{F}_p$-dimensions of certain local cohomology groups which will be introduced in section 4. This is accomplished using a result of Mazur and Rubin which uses Cassels-Poitou-Tate global duality. Indeed, using this we obtain that
$$
\dim_{\mathbb{F}_p} \Sel_p(E/L)^G = \sum_v \delta_v
$$
where the $\delta_v$ are the $\mathbb{F}_p$-dimensions of the above mentioned local cohomology groups. Section 5 is almost entirely devoted to computing the $\mathbb{F}_p$-dimensions of these local cohomology groups which in many cases is equivalent to computing local norm indices. Of course the $\delta_v$ depend heavily on the type of reduction of $E$ at $v$ as well as the splitting behavious of $v$ in $L$. In section 3 we study the local norm map of formal groups in cyclic Galois extensions, as this is used in section 5 to bound the $\delta_v$ coming from supersingular primes.  Theorem \ref{MainThm2} is proved in section 5, along with Corollary \ref{MainCor}. As mentioned earlier, Theorem \ref{MainThm1} essentially follows from Theorem \ref{MainThm2}, except that because of the nature of the extensions considered there we are able to obtain more precise information. In section 6 we finish the proof of Theorem \ref{MainThm1}.

\vspace{5 mm}

\noindent \textbf{Acknowledgements.} I would like to thank Tim Dokchitser and Peter Stevenhagen for the constant guidance and support, as well as many helpful discussions and suggestions. I am also grateful to Vladimir Dokchitser for many discussions regarding the material of the paper. I also thank Alex Bartel, Kestutis \v{C}esnavi\v{c}ius and Adam Morgan for helpful conversations.

\section{Background and notation}

In this section we recall some properties of Selmer groups of elliptic curves (see e.g. \cite{Sil}, \S X), as well as establish some notation.
\medskip

If $K$ is a number field we will write $G_K$ for the absolute Galois group $\Gal(\overline{K}/K)$. For $G$ a profinite group and $A$ a discrete $G$-module, we write $H^n(G,A)$ to denote the cohomology groups of $G$ formed with continuous cochains. As usual, if $A$ is a $G_K$-module, we write $H^n(K,A)$ to denote the cohomology group $H^n(G_K,A)$, and sometimes we will write $H^n(L/K,A)$ to denote $H^n(\Gal(L/K),A)$. We will use $\inf,\res$ and $\cores$ to denote the inflation, restriction and corestriction maps of cohomology, respectively. For $M$ an abelian group and $n$ a positive integer, we write $M[n]$ to denote the $n$-torsion subgroup of $M$, that is, the subgroup of elements of $M$ annihilated by $n$. If $M$ is a torsion abelian group and $p$ is a prime, we denote by $M[p^\infty]=\cup_n M[p^n]$ the $p$-primary component of $M$. 

\medskip

As mentioned in the introduction, the $n$-Selmer group of $E$ over $K$ is defined by
$$
\Sel_n(E/K)=\ker\Big(H^1(K,E[n])\longrightarrow \prod_v H^1(K_v,E(\overline{K_v}))\Big).
$$
The $n$-Selmer group fits into the exact sequence

$$
0 \longrightarrow E(K)/nE(K) \longrightarrow \Sel_n(E/K) \longrightarrow \Sha(E/K)[n] \longrightarrow 0.
$$

\begin{pro}
Let $p$ be a prime let $S$ be a finite set of primes of $K$ containing all archimidean primes, all primes dividing $p$ and all primes where $E$ has bad reduction. Then
$$
\Sel_p(E/K)=\ker\Big(H^1(K_S/K,E[p])\longrightarrow \prod_{v \in S} H^1(K_v,E(\overline{K_v}))[p]\Big).
$$
where $K_S$ denotes the maximal extension of $K$ unramified outside $S$.
\end{pro}

\begin{proof}
This is \cite{Milne} \S I, Corollary 6.6.
\end{proof}

\begin{pro}[Cassels-Poitou-Tate]\label{CPT}
Let $E$, $S$ and $K$ be as above. Then there is an exact sequence
$$
\begin{CD}
\Sel_p(E/K) @>>> H^1(K_S/K,E[p]) @>>>\prod_{v\in S} H^1(K_v,E(\overline{K_v}))[p] \\
   @.										@.												@VVV									\\
\prod_{v\in S} H^2(K_v,E(\overline{K_v}))[p]    @<<<      H^2(K_S/K,E[p])  @<<< \widehat{\Sel_p(E/K)}\\
    @VVV \\
    \widehat{E(K)[p]}  @>>> 0.
\end{CD}
$$
\end{pro}

\begin{proof}
See for instance \cite{GCEC} \S 1.7.
\end{proof}

\begin{pro}\label{pdiv}
Let $G$ be a pro-$p$-group and $A$ a $G$-module which is uniquely divisible by $p$. Then
$$
H^i(G,A)=0 
$$
for all $i\geq 1$.
\end{pro}

\begin{proof}
The exact sequence
$$
0 \longrightarrow A[p] \longrightarrow A \longrightarrow A \longrightarrow 0
$$
yields an exact sequence of cohomology groups
$$
H^n(G,A[p]) \longrightarrow H^n(G,A) \overset{p}{\longrightarrow} H^n(G,A) \longrightarrow H^{n+1}(G,A[p]).
$$
Since $A$ is uniquely divisible by $p$ we have that $A[p]$ is trivial and so we obtain an isomorphism
$$
[p]:H^n(G,A) \overset{\sim}{\longrightarrow} H^n(G,A),
$$
hence $H^n(G,A)$ is uniquely divisible by $p$ for all $n$. But $G$ is a pro-$p$-group and so every element of $H^n(G,A)$ is annihilated by some power of $p$ for $n \geq 1$ so the result follows.
\end{proof}

\section{The norm map on formal groups}

 In this section we will analyse the cokernel of the norm map of formal groups in cyclic degree $p$ extensions.
 
\subsection{Generalities on formal groups}
For this section we let $K$ be a finite extension of $\mathbb{Q}_p$ with residue field $k$ and normalized valuation $v_K$. We denote its ring of integers by $\mathcal{O}_K$ and its maximal ideal by $\mathfrak{m}_K$. Let $\mathcal{F}$ be a one-dimensional commutative formal group law defined over $\mathcal{O}_K$. Recall this is given by a power series $F\in\mathcal{O}_K[[X,Y]]$ which satisfies the following properties:
\begin{itemize}
\item[(i)] $F(X,Y)=F(Y,X)$
\item[(ii)] $ F(X,0)=F(0,X)=X$
\item[(iii)] $F(X,F(Y,Z))=F(F(X,Y),Z)$
\end{itemize}
As usual, we define a group operation on the set $\mathfrak{m}_K$ by
$$
x\oplus_\mathcal{F} y = F(x,y)
$$
and we denote the corresponding abelian group by $\mathcal{F}(\mathfrak{m}_K)$. When there is no risk of confusion we will omit the subscript $\mathcal{F}$ from $\oplus_\mathcal{F}$.

Define  $F_n(X_1,\dots X_n)\in \mathcal{O}_K[[X_1,\dots, X_n]]$ inductively by 
$$
F_2(X_1,X_2)=F(X_1,X_2), \quad F_{n+1}(X_1,\dots,X_{n+1})=F(F(X_1,\dots, X_n),X_{n+1}).
$$
Then given a positive integer $n$, the multiplication by $n$ map on $\mathcal{F}(\mathfrak{m}_K)$ is simply given by the power series $[n]_{\mathcal{F}}(X)=F_n(X,\dots,X)$. Consider the multiplication-by-$p$ map $[p]_{\mathcal{F}}(X)=F_p(X,\dots,X)$ mod $\mathfrak{m}_K$. Recall that (see \cite{Sil}) the \emph{height} of $\mathcal{F}$, denoted $\text{ht}(\mathcal{F})$ is the largest integer $h$ such that $[p]_{\mathcal{F}}(X)\equiv g(X^{p^h}) \mod\mathfrak{m}_K$ for some power series $g(X)\in \mathcal{O}_K[[X]]$, where we set $h=\infty$ if $[p]_{\mathcal{F}}(X)\equiv 0 \mod\mathfrak{m}_K$.

\subsection{The norm map of a formal group}
Let $L$ be a finite Galois extension of $K$ with Galois group $G=\{\sigma_1,\dots,\sigma_n\}$. Then we also have an abelian group structure on $\mathfrak{m}_L$ defined using $F$ as above, and we denote this group by $\mathcal{F}(\mathfrak{m}_L)$. Recall that there is natural action of $G$ on $\mathcal{F}(\mathfrak{m}_L)$ and that the norm map is defined by
\begin{align*}
&N^{\mathcal{F}}_{L/K}:\mathcal{F}(\mathfrak{m}_L) \longrightarrow \mathcal{F}(\mathfrak{m}_K)\\
&x \longmapsto \sigma_1(x)\oplus \dots \oplus \sigma_n(x).
\end{align*}
The following lemma, which appeared in \cite{Haze}, will be quite useful when studying the norm map on formal groups. In order to state it we need some notation. Given a monomial $M=X_1^{r_1}\dots X_n^{r_n}$ in the variables $X_1,\dots,X_n$, define
$$
\Tr(M)=X_1^{r_1}\dotsb X_n^{r_n}+X_2^{r_1}\dotsb X_n^{r_{n-1}}X_1^{r_n}+\dots+X_n^{r_1}\dotsb X_{n-1}^{r_n}.
$$

\begin{lem}\label{HazeLem}
Let $(\mathcal{F},F)$ be a formal group as above with height $h<\infty$. Then
\begin{equation}\label{Fnorm}
F_n(X_1,\dots,X_n)=\Tr(X_1)+\sum_{i=1}^\infty a_i(X_1 X_2\dotsb X_n)^i+\sum_M a_M\Tr(M)
\end{equation}
where $M$ runs over all monomials of degree at least $2$ which are not of the form $(X_1\dotsb X_n)^i$. If in addition $n=p$, then $v_K(a_i)\geqslant 1$ unless $i=kp^{h-1}$ with $k=1,2\dots$, and $v_K(a_i)=0$ when $i=p^{h-1}$.
\end{lem}

\begin{proof}
Equation (\ref{Fnorm}) follows from the fact that $F_n(X_1,\dots,X_n)=F_n(X_{\sigma(1)},\dots,X_{\sigma(n)})$ for every permutation $\sigma$ of $\{1,\dots,n\}$. The second statement follows directly from the definition of height, since reducing mod $\mathfrak{m}_K$ gives something of the form $g(X^{p^h})$.
\end{proof}

\begin{cor}\label{NormCor}
Let $L/K$ be a finite Galois extension of degree $n$, and let $\Tr_{L/K}$ and $N_{L/K}$ denote the usual trace and field norm. Then for all $x\in \mathcal{F}(\mathfrak{m}_L)$ we have
$$
N_{\mathcal{F}}(x) \equiv \Tr_{L/K}(x)+\sum_{i=1}^\infty a_i N_{L/K}(x)^i \mod \Tr_{L/K}(x^2\mathcal{O}_L).
$$
If $n=p$ then the $a_i$ satisfy the same statements as in Lemma \ref{HazeLem}.
\end{cor}

\begin{proof}
By definition we have $N^{\mathcal{F}}_{L/K}(x)=F_n(\sigma_1(x),\dots, \sigma_n(x))$, so the result follows from Lemma \ref{HazeLem}.
\end{proof}

\subsection{Norm map of formal groups in cyclic extensions}

As before $K$ is a finite extension of $\mathbb{Q}_p$ and $(\mathcal{F},F)$ is a one-dimensional commutative formal group law over $\mathcal{O}_K$. Denote by $f_K$ the inertia degree of $K$ over $\mathbb{Q}_p$, that is, the dimension of $k$ as an $\mathbb{F}_p$-vector space. Let $L/K$ be a cyclic totally ramified Galois extension with $G\simeq \mathbb{Z}/p\mathbb{Z}$. Let $\sigma$ be a generator of $G$ and put $t=v_L(\sigma(\pi_L)-\pi_L)-1$. Then the ramification groups of $G$ are
\begin{align*}
G&=G_0=\dots =G_t\\
\{1\}&=G_{t+1}=\dotsb
\end{align*}
where $G_0$ is the inertia group  of $L/K$.
Since $L/K$ is totally ramified of degree $p$, it is wildly ramified and hence $t\geqslant 1$. Denote by $\mathfrak{D}=\mathfrak{D}_{L/K}$ the different, and let $m$ be defined by the equation
$$
\mathfrak{D}=\mathfrak{m}_L^m.
$$
Then it is known (see for example \cite{SerreLF}, \S IV.3) that $m=(t+1)(p-1)$. Note since $L/K$ is totally ramified, $m$ also equals $\ord_K(\Delta_{L/K})$  and we have the following lemma.

\begin{lem}\label{TraceLemma}
Let $n\geqslant 0$ be an integer, and set $r=\lfloor (m+n)/p \rfloor$. Then 
$$
\Tr_{L/K} (\mathfrak{m}_L^n)=\mathfrak{m}_K^r.
$$
\end{lem}

\begin{proof}
Since the trace is $K$-linear, it follows that $\Tr_{L/K} (\mathfrak{m}_L^n)$ is an ideal of $\mathcal{O}_K$. Also, by definition of the different we have 
$$
\Tr_{L/K} (\mathfrak{m}_L^n) \subset \mathfrak{m}_K^r \quad \Longleftrightarrow \quad\mathfrak{m}_L^n \subset \mathfrak{D}^{-1}\mathfrak{m}_K^r=\mathfrak{m}_L^{pr-m} \quad \Longleftrightarrow \quad r\leqslant (m+n)/p.
$$
\end{proof}

We now give a necessary and sufficient condition for the norm map to be surjective in the case where the height of the formal group is at least $2$. For this we will need the following lemma, which is essentially the same as Lemma 2 in \cite{SerreLF} \S V.1.

\begin{lem}\label{Filtration}
Let $A,B$ be abelian groups filtered by subgroups $A=A_1\supset A_2\supset\dotsb$, $B=B_1 \supset B_2\supset \dotsb$ such that $A$ is complete with respect to the topology defined by the $A_n$ and $\bigcap B_n=\{0\}$. Let $u:A\rightarrow B$ be a homomorphism and suppose there exist indices $t_1<t_2<\dotsb$ such that $u(A_{t_i})\subset B_i$ and the induced maps $A_{t_i}\rightarrow B_i/B_{i+1}$ are surjective for all $i\geqslant 1$. Then $u$ is surjective.
\end{lem}

\begin{proof}
Let $b\in B_1$. Then there exists $a_1\in A_{t_1}$ and $b_2\in B_2$ such that $u(a_1)=b-b_2$. By the same argument there exists $a_2\in A_{t_2}$ and $b_3\in B_3$ such that $u(a_2)=b_2-b_3$. Continuing in this manner we obtain a sequence $a_1+a_2+\dotsb$ which converges to $a\in A$ since $A$ is complete. For each $n>1$ we have that $u(a_1+a_2+\dotsb+a_{n-1})-b=-b_n\in B_n$ and $a'=\sum_{i=n}^\infty a_i\in A_{t_n}$, so it follows that
\begin{align*}
u(a)-b&=u(a_1+\dotsb + a_{n-1}+a')-b\\
          &=u(a_1+\dotsb + a_{n-1})-b+u(a')
\end{align*}
is in $B_n$. We conclude $u(a)-b\in \bigcap B_n={0}$, hence $u(a)=b$.
\end{proof}

\begin{pro}\label{SurjectiveNorm}
Suppose $\mathcal{F}$ has height $h>1$, and let $t=v_L(\sigma(\pi_L)-\pi_L)-1$. Then
$$
N^{\mathcal{F}}_{L/K}(\mathcal{F}(\mathfrak{m}_L)) = \mathcal{F}(\mathfrak{m}_K) \Longleftrightarrow t=1.
$$

Moreover, for $t\geqslant 2$ we have 
$$
f_K \leqslant \dim_{\mathbb{F}_p} \mathcal{F}(\mathfrak{m}_K)/N^{\mathcal{F}}_{L/K}(\mathcal{F}(\mathfrak{m}_L)) \leqslant \min\{f_K(t-1), [K:\mathbb{Q}_p]+h\} 
$$
\end{pro}

\begin{proof}

We show first that $\mathcal{F}(\mathfrak{m}^t_L)$ surjects via the norm onto $\mathcal{F}(\mathfrak{m}^t_K)$. Note that by Lemma \ref{TraceLemma} we have $\Tr_{L/K}(\mathfrak{m}_L^t)=\mathfrak{m}_K^t$ and $\Tr_{L/K}(\mathfrak{m}_L^{t+1})=\mathfrak{m}_K^{t+1}$. Hence we may set $t_i=(i-1)p+t$ for $i \geqslant 1$. Then again by the lemma the $t_i$ satisfy

\begin{itemize}
\item[(i)]
$\Tr_{L/K}(\mathfrak{m}_L^{t_i})=\mathfrak{m}_K^{t+i-1}$
\item[(ii)]
$\Tr_{L/K}(\mathfrak{m}_L^{t_i+1})=\mathfrak{m}_K^{t+i}$.
\end{itemize}
Using Corollary \ref{NormCor} and the fact that $h>1$ we obtain
$$
N_\mathcal{F}(x)\equiv \Tr_{L/K}(x) \mod \mathfrak{m}_K^{t+i}
$$
for all $x$ in $\mathcal{F}(\mathfrak{m}_L^{t_i})$. It follows that the induced maps $\mathcal{F}(\mathfrak{m}_L^{t_i}) \rightarrow \mathcal{F}(\mathfrak{m}_K^{t+i-1})/\mathcal{F}(\mathfrak{m}_K^{t+i})$ are surjective for all $i\geqslant 1$, and the claim follows from Lemma \ref{Filtration}.

From this we immediately obtain that the norm map is surjective for $t=1$. Now suppose that $t\geqslant 2$. Then keeping the above notation we have $m\geqslant 3p-3$ and hence by Lemma \ref{TraceLemma} we have $\Tr_{L/K}(\mathfrak{m}_L)\subset \mathfrak{m}_K^2$. Let $x\in \mathcal{F}(\mathfrak{m}_L)$. Then by Corollary \ref{NormCor} and using that $h>1$ we conclude $N_\mathcal{F}(x) \in \mathcal{F}(\mathfrak{m}_K^2)$, hence $N^{\mathcal{F}}_{L/K}(\mathcal{F}(\mathfrak{m}_L)) \subset \mathcal{F}(\mathfrak{m}_K^2)$. It follows then that
$$
\dim_{\mathbb{F}_p} \mathcal{F}(\mathfrak{m}_K)/\mathcal{F}(\mathfrak{m}_K^2) \leqslant \dim_{\mathbb{F}_p} \mathcal{F}(\mathfrak{m}_K)/N^{\mathcal{F}}_{L/K}(\mathcal{F}(\mathfrak{m}_L)).
$$
But 
$$
\mathcal{F}(\mathfrak{m}_K)/\mathcal{F}(\mathfrak{m}_K^2) \simeq \mathfrak{m}_K/\mathfrak{m}_K^2,
$$
which is a vector space over the residue field $k$ of degree $1$, hence the left hand side of the inequality. For the right hand side, note that from the claim it follows that $N^{\mathcal{F}}_{L/K}(\mathcal{F}(\mathfrak{m}_L)) \supset \mathcal{F}(\mathfrak{m}_K^t)$ and hence
$$
\dim_{\mathbb{F}_p} \mathcal{F}(\mathfrak{m}_K)/N^{\mathcal{F}}_{L/K}(\mathcal{F}(\mathfrak{m}_L)) \leqslant \dim_{\mathbb{F}_p} \mathcal{F}(\mathfrak{m}_K)/\mathcal{F}(\mathfrak{m}_K^t)
$$
which has $\mathbb{F}_p$-dimension $f_K(t-1)$ by the same argument as before. Finally, recall that for $K$ a finite extension of $\mathbb{Q}_p$, we have an isomorphism of abelian groups
$$
\mathcal{F}(\mathfrak{m}_K) \simeq \mathbb{Z}_p^{[K:\mathbb{Q}_p]}\times T
$$
where $T$ is a finite group such that $T/pT$ has $\mathbb{F}_p$-dimension at most $h$. The result then follows from the fact that $p\mathcal{F}(\mathfrak{m}_K) \subset N^{\mathcal{F}}_{L/K}(\mathcal{F}(\mathfrak{m}_L))$.
\end{proof}

Note that since $G$ is a cyclic, $\mathcal{F}(\mathfrak{m}_K)/N^{\mathcal{F}}_{L/K}(\mathcal{F}(\mathfrak{m}_L))$ is equal to $H^2(G,\mathcal{F}(\mathfrak{m}_L))$. We now end this section with a result which will be useful when computing Selmer ranks of elliptic curves.

\begin{pro}\label{herb}
Suppose $\mathcal{F}$ and $L/K$ are as in this section. Then 
$$
H^1(G,\mathcal{F}(\mathfrak{m}_L)) \simeq H^2(G,\mathcal{F}(\mathfrak{m}_L)).
$$
\end{pro}

\begin{proof}
Because both groups are $\mathbb{F}_p$-vector spaces, it suffices to show that the Herbrand quotient
$$
h_G(\mathcal{F}(\mathfrak{m}_{L}))=\frac{\# H^2(G,\mathcal{F}(\mathfrak{m}_{L}))}{\# H^1(G,\mathcal{F}(\mathfrak{m}_{L}))}
$$
is equal to 1.
In order to show this, recall that for large enough $n$ the formal logarithm induces an isomorphism $\mathcal{F}(\mathfrak{m}_{L}^n)\overset{\sim}{\longrightarrow} \widehat{\mathbb{G}}_a(\mathfrak{m}_{L}^n)$ with the additive formal group on $\mathfrak{m}_{L}$. This isomorphism is $G$-equivariant since $\mathcal{F}$ is defined over $K$. Thus we have a short exact sequence of $G$-modules
$$
0\longrightarrow \widehat{\mathbb{G}}_a(\mathfrak{m}_{L}^n) \longrightarrow \mathcal{F}(\mathfrak{m}_{L}) \longrightarrow C \longrightarrow 0
$$
where $C$ is finite. By a well-known property (see for example \cite{SerreLF}, \S VIII.4) of the Herbrand quotient one has 
$$
h_G\big(\mathcal{F}(\mathfrak{m}_{L})\big)=h_G\big(\widehat{\mathbb{G}}_a(\mathfrak{m}_{L}^n)\big)h_G(C).
$$
Finally, we know that for $n$ divisible by $p$ we have $\widehat{\mathbb{G}}_a(\mathfrak{m}_{L}^n) \simeq \mathcal{O}_{L}$ as $G$-modules and $\mathcal{O}_{L}$ is an induced $G$-module, hence $h_G\big(\widehat{\mathbb{G}}_a(\mathfrak{m}_{L}^n)\big)=1$. Also $h_G(C)=1$ as $C$ is finite. We conclude then that $h_G(\mathcal{F}(\mathfrak{m}_{L}))=1$, and this completes the proof.
\end{proof}

\section{Twisted Selmer groups of elliptic curves}

Let $K$ be a number field and $p$ an odd prime. For $E$ an elliptic curve defined over $K$, recall that $E[p]$ is a 2-dimensional $\mathbb{F}_p$-vector space with a continuous $G_K$-action and with a perfect, skew-symmetric, $G_K$-equivariant self-duality
$$
E[p] \times E[p] \longrightarrow \mathbb{\mu}_p.
$$
A \emph{Selmer structure} $\mathcal{F}$ on $E[p]$ is a collection of subspaces $H^1_\mathcal{F}(K_v,E[p])$ for every prime $v$ of $K$ such that for all but finitely many $v$ we have $H^1_\mathcal{F}(K_v,E[p])=H^1(K_v^{\mathrm{ur}}/K,E[p]^{I_v})$ where $K_v^{\mathrm{ur}}$ denotes the maximal unramified extension of $K_v$ and $I_v$ is the inertia group. Given Selmer structures $\mathcal{F}$ and $\mathcal{G}$ on $E[p]$, define Selmer structures $\mathcal{F}+\mathcal{G}$ and $\mathcal{F}\cap\mathcal{G}$ by
\begin{align*}
H^1_{\mathcal{F}+\mathcal{G}}(K_v,E[p]) &= H^1_\mathcal{F}(K_v,E[p]) + H^1_\mathcal{G}(K_v,E[p])\\
H^1_{\mathcal{F}\cap\mathcal{G}}(K_v,E[p]) &= H^1_\mathcal{F}(K_v,E[p])\cap H^1_\mathcal{G}(K_v,E[p]).
\end{align*}
Given a Selmer structure	$\mathcal{F}$ on $E[p]$ define the Selmer group by
$$
H^1_\mathcal{F}(K,E[p]) = \ker\Big(H^1(K,E[p]) \longrightarrow \prod_v H^1(K_v,E[p])/H^1_\mathcal{F}(K_v,E[p])\Big).
$$

\begin{exa}
For each $v$ in $K$ let $H^1_\mathcal{F}(K_v,E[p]) = \mathrm{im}(\lambda_{E,v})$, where 
$$
\lambda_{E,v} : E(K_v)/pE(K_v) \longrightarrow H^1(K_v,E[p])
$$
is the Kummer map. Then by Lemma 19.3 of \cite{Cassels}, $H^1_\mathcal{F}(K_v,E[p]) = H^1(K_v^{\mathrm{ur}}/K,E[p])$ for all $v$ such that $v \nmid p$ and $E$ has good reduction at $v$. With this Selmer structure given on $E[p]$ we have that $H^1_\mathcal{F}(K,E[p]) = \Sel_p(E/K)$ is the usual $p$-Selmer group as defined in section 1.
\end{exa}

Let $L/K$ be a Galois extension of number fields of degree $p$ with Galois group $G$ such that $G$ is generated by $\sigma$. Denote by $\Res{L}{K}{E}$ the Weil restriction of scalars of $E$ from $L$ to $K$. Then for every $K$-algebra $X$ there is an isomorphism, functorial in $X$,
\begin{equation}\label{ResIsom}
(\Res{L}{K}{E})(X) \simeq E(X\otimes_K L).
\end{equation}
The action of $G$ on $E(X\otimes_K L)$ induces a canonical inclusion $\mathbb{Z}[G] \hookrightarrow \End_K(\Res{L}{K}{E})$ (see for instance section 4 of \cite{MRS}).
Let $\mathcal{O}$ be the ring of integers of the cyclotomic field of $p$th roots of unity, and $\mathfrak{p}$ the maximal ideal of $\mathcal{O}$. Let $A$ be the twist of $E$ by the irreducible rational representation of $G$ corresponding to $L$. Hence $A$ is the abelian variety denoted $E_L$ in Definition 5.1 of \cite{MRS}. If we let
$$
N_{L/K} = \sum_{g\in G} g \in \mathbb{Z}[G]
$$
then it follows from Proposition 4.2 in \cite{MRS} that
$$
E = \ker(\sigma - 1) \subset \Res{L}{K}{E}, \quad A = \ker (N_{L/K}) \subset \Res{L}{K}{E}.
$$
Hence, $A$ is an abelian variety of dimension $p-1$ over $K$, and it is the kernel of the map $N_{L/K}: \Res{L}{K}{E} \rightarrow E$.
We also have from Theorem 3.4 of \cite{MazurRubin} that the following holds:
\begin{itemize}
\item[(i)]
The inclusion $\mathbb{Z}[G]\hookrightarrow \End_K(\Res{L}{K}{E})$ induces a ring homomorphism $\mathbb{Z}[G]\rightarrow \End_K(A)$ that factors
$$
\mathbb{Z}[G] \twoheadrightarrow \mathcal{O} \hookrightarrow \End_K(A)
$$
where the first map is induced by the projection in (3.2) of \cite{MazurRubin}.
\item[(ii)]
For every commutative $K$-algebra $X$, the isomorphism of (\ref{ResIsom}) restricts to an isomorphism, functorial in $X$,
$$
A(X) \simeq \{ x\in E(X\otimes_K L) : \sum_{g\in G} (1\otimes g)(x) = 0 \}.
$$
\end{itemize}

\begin{pro}\label{IsomOfTorsion}
There is a canonical $G_K$-isomorphism $A[\mathfrak{p}] \xrightarrow{\sim} E[p]$.
\end{pro}

\begin{proof}
This is Proposition 4.1 of \cite{MazurRubin}.
\end{proof}

Using Proposition \ref{IsomOfTorsion} we may define another Selmer structure $\mathcal{A}$ on $E[p]$ as follows. Let $\pi$ be a generator for $\mathfrak{p}$. For every place $v$ of $K$ let 
$$
\lambda_{A,v}:A(K_v)/\pi A(K_v) \hookrightarrow H^1(K_v,A[\mathfrak{p}]) 
$$
denote the Kummer map and set $H^1_{\mathcal{A}}(K_v,E[p])$ to be the image of $\lambda_{A,v}$ composed with the isomorphism of Proposition \ref{IsomOfTorsion}. For each place $w$ of $L$, let $H^1_{\mathcal{F'}}(L_w,E[p])$ be the image of the Kummer map
$$
E(L_w)/pE(L_w) \hookrightarrow H^1(L_w,E[p]).
$$

\begin{lem}\label{resKernel}
Let $v$ be a prime of $K$, and $w$ a prime of $L$ above $v$. If $\res$ denotes the restriction map 
$$
H^1(K_v,E[p]) \longrightarrow H^1(L_w,E[p])
$$ 
then we have that $\ker (\res) \subset H^1_{\mathcal{F}+\mathcal{A}}(K_v,E[p])$.
\end{lem} 

\begin{proof}
Note that there is nothing to prove if $v$ splits completely in $L$, so suppose there is one prime $w$ of $L$ above $v$. It is equivalent to show that the image of the inflation map 
$$
H^1(L_w/K_v, E(L_w)[p]) \longrightarrow H^1(K_v,E[p])
$$
is contained in $H^1_{\mathcal{F}+\mathcal{A}}(K_v,E[p])$. The short exact sequence
$$
0\longrightarrow E[p] \longrightarrow (\Res{L}{K}{E})[p] \xrightarrow{\sigma-1} A[p] \longrightarrow 0
$$
induces the exact sequence in cohomology
\begin{equation}\label{ResCohom}
0 \longrightarrow \frac{A(K_v)[p]}{(\sigma-1)\big((\Res{L}{K}{E})(K_v)[p]\big)}  \xrightarrow{\ f\ } H^1(K_v,E[p]) \xrightarrow{\ g\ } H^1(K_v,(\Res{L}{K}{E})[p]).
\end{equation}
Note that $(\Res{L}{K}{E})(K_v) \simeq E(K_v\otimes L) = E(L_w)$, hence we may identify $A(K_v)$ with the elements $x\in E(L_w)$ such that $ N_{L_w/K_v}(x)=0\}$. We have then isomorphisms
\begin{align*}
\frac{A(K_v)[p]}{(\sigma-1)\big((\Res{L}{K}{E})(K_v)[p]\big)} & \simeq H^1(L_w/K_v,E(L_w)[p]), \\
H^1(K_v,(\Res{L}{K}{E})[p]) &\simeq H^1(L_w,E[p]) 
\end{align*}
where the first isomorphism follows from $G$ being cyclic and the second one is by Shapiro's lemma (see for instance the proof of Proposition 3.1 in \cite{MazurRubin}). Using these identifications one sees that indeed the maps $f$ and $g$ in (\ref{ResCohom}) are the usual inflation and restriction maps. We now show that the image of $f$ is contained in $H^1_{\mathcal{F}+\mathcal{A}}(K_v,E[p])$. To this end, let $\alpha\in\mathcal{O}^\times$ be such that $\alpha \pi^{p-1}=p$. Let $P \in A(K_v)[p]$, and let $Q_1\in A(\overline{K_v}) $ be such that $\pi(Q_1)=P$. Set $P'=\alpha\pi^{p-2}P$. Then $P'\in A(K_v)$ hence $(\sigma-1)$ acts as $\pi$ on it, so we obtain
$$
(\sigma-1)P' = \pi(P')=pP = 0
$$
and it follows that $P'\in E(K_v)$. Now let $Q_2 \in E(\overline{K_v})$ such that $pQ_2 = P'$. Finally, set $Q=Q_1-Q_2 \in (\Res{L}{K}{E})(\overline{K_v})[p]$. Then as $Q_2\in E$ we have that $(\sigma-1)(Q_2)=0$ and so $(\sigma -1)(Q) = P$. It follows that the image of $P$ under $f$ is represented by the cocycle $\{Q^g-Q\} = \{(Q_1^g-Q_1)\ - (Q_2^g-Q_2)\}$. This cocycle also represents $\lambda_{A,v}(P)-\lambda_{E,v}(\alpha\pi^{p-2}P)$ which is contained in $H^1_{\mathcal{F}+\mathcal{A}}(K_v,E[p])$ and this completes the proof. 
\end{proof}

\begin{pro}\label{resStruct}
Let $v$ be a prime of $K$ and $w$ a prime of $L$ above $v$. Then 
$$
\res^{-1}\big(H^1_{\mathcal{F'}}(L_w,E[p])\big) = H^1_{\mathcal{F}+\mathcal{A}}(K_v,E[p])
$$
\end{pro}

\begin{proof}
The inclusion $\res^{-1}\big(H^1_{\mathcal{F'}}(L_w,E[p])\big) \supset H^1_{\mathcal{F}+\mathcal{A}}(K_v,E[p])$ follows from considering the following commutative diagram with exact rows
$$
\begin{CD}
0 @>>> E(K_v)/pE(K_v)  @>\lambda_{E,v} >> H^1(K_v,E[p]) \\
@.               @VVV                                   @VV\res V         \\
0 @>>> E(L_w)/pE(L_w)  @>\lambda_{E,w} >> H^1(L_w,E[p])
\end{CD}
$$
and the corresponding one for $A$,
$$
\begin{CD}
0 @>>> A(K_v)/\pi A(K_v)  @>\lambda_{A,v} >> H^1(K_v,E[p]) \\
@.               @VV\pi^{p-2}V                                   @VV\res V         \\
0 @>>> E(L_w)/pE(L_w)  @>\lambda_{E,w} >> H^1(L_w,E[p]).
\end{CD}
$$
For the reverse inclusion, we again suppose there is only one prime $w$ of $L$ above $v$, as otherwise there is nothing to prove. Take $x\in \res^{-1}\big(H^1_{\mathcal{F'}}(L_w,E[p])\big)$, so that $\res(x)\in H^1_{\mathcal{F'}}(L_w,E[p])$.  Then $\res(x) = \lambda_{E,w}(P)$ for some $P \in E(L_w)$. From the commutative diagram
$$
\begin{CD}
0 @>>> E(K_v)/pE(K_v)  @>\lambda_{E,v} >> H^1(K_v,E[p]) \\
@.               @AA N_{L_w/K_v} A                                   @AA \cores A        \\
0 @>>> E(L_w)/pE(L_w)  @>\lambda_{E,w} >> H^1(L_w,E[p])
\end{CD}
$$
we obtain that
\begin{align*}
\lambda_{E,v}\big(N_{L_w/K_v}(P)\big) &= \cores\big(\lambda_{E,w}(P)\big) \\
															 &= \cores(\res(x)) \\
															 &= 0.
\end{align*}
It follows that $N_{L_w/K_v}(P)=pQ$ for some $Q\in E(K_v)$. Finally, set $R = P-Q$. Then $N_{L_w/K_v}(R)=0$ hence we may view $R$ as an element of $A(K_v)$. It follows that $\lambda_{A,v}(R)+\lambda_{E,v}(Q)-x$ is in contained in the kernel of $\res$, and the result now follows from Lemma \ref{resKernel}.
\end{proof}

\section{Growth in Selmer groups}

In this section we use the results of previous sections to study the growth of $p$-Selmer groups in cyclic extensions. Assume for now that $L/K$ is a finite Galois extension of number fields and $E/K$ is an elliptic curve over $K$. For $v$ a (archimidean or non-archimidean) prime of $K$ we define $W_{v,L}$ by
$$
W_{v,L}=\ker\Big(H^1(K_v,E(\overline{K_v}) \longrightarrow H^1(L_w,E(\overline{L_w}))\Big),
$$
where $w$ is any place of $L$ lying above $v$. As $L/K$ is Galois, this definition is independent of the choice of $w$. Also, if we set $G_v=\Gal(L_w/K_v)$ then by inflation-restriction we have
$$
W_{v,L} \simeq H^1(G_v,E(L_w)).
$$
It follows from this that $W_{v,L}$ is finite.

Suppose now that $L/K$ has odd prime degree $p$, and let $G=\Gal(L/K)$. Let $S$ be a finite set of primes of $K$ containing all archimidean primes, all primes above $p$, primes of bad reduction of $E$ and all primes ramified in $L/K$. 

\begin{lem}\label{WvLdim}
For every prime $v$ in $K$, we have $W_{v,L} \simeq H^1_{\mathcal{F}+\mathcal{A}}(K_v,E[p])/H^1_{\mathcal{F}}(K_v,E[p])$.
\end{lem}

\begin{proof}
Consider the commutative diagram with exact rows
$$
\begin{CD}\label{ComDia}
0 @>>> E(K_v)/pE(K_v) @>\lambda_{E,v} >> H^1(K_v,E[p]) @>\beta >> H^1(K_v,E) \\
@.                  @VVV                                         @V\res VV                @V \res' VV         \\      
0 @>>> E(L_w)/pE(L_w) @>\lambda_{E,w} >> H^1(L_w,E[p]) @>>> H^1(L_w,E) .
\end{CD}
$$
It is not hard to see $\beta$ induces a map from $H^1_{\mathcal{F}+\mathcal{A}}(K_v,E[p])$ to $W_{v,L}$. Indeed, if $x$ is in $H^1_{\mathcal{F}+\mathcal{A}}(K_v,E[p])$ then by Proposition \ref{resStruct} we have that $\res'(\beta(x))=0$, hence $\beta(x)\in \ker(\res')=W_{v,L}$. It is easy to see that this map is onto and its kernel is $H^1_\mathcal{F}(K_v,E[p])$.
\end{proof}

\begin{pro}\label{mainPro}
If $\Sel_p(E/K)$ is trivial, then 
\begin{equation*}
\dim_{\mathbb{F}_p} \Sel_p(E/L)^{\Gal(L/K)} = \sum_{v \in S} \dim_{\mathbb{F}_p} W_{v,L}
\end{equation*}
\end{pro}

\begin{proof}
Using Propositions 1.3, 2.1 and 4.4 of \cite{MazurRubin} we obtain 
\begin{align*}
\dim_{\mathbb{F}_p} H^1_{\mathcal{F}+\mathcal{A}}(K,E[p])/H^1_{\mathcal{F}\cap\mathcal{A}}(K,E[p]) &= \sum_{v\in S} \dim_{\mathbb{F}_p} H^1_{\mathcal{A}}(K_v,E[p])/H^1_{\mathcal{F}\cap\mathcal{A}}(K_v,E[p])\\
									& = \sum_{v\in S} \dim_{\mathbb{F}_p} H^1_{\mathcal{F} + \mathcal{A}}(K_v,E[p])/H^1_{\mathcal{F}}(K_v,E[p])\\
									& = \sum_{v\in S} \dim_{\mathbb{F}_p} W_{v,L}
\end{align*}
where the third equality follows from Lemma \ref{WvLdim}. Now if $\Sel_p(E/K)$ is trivial, then so is $H^1_{\mathcal{F}\cap\mathcal{A}}(K,E[p])$. Also, by Proposition \ref{resStruct} we have that $H^1_{\mathcal{F}+\mathcal{A}}(K,E[p]) = \res^{-1}(\Sel_p(E/L))$ and $\res$ induces an isomorphism $H^1_{\mathcal{F}+\mathcal{A}}(K_v,E[p]) \simeq \Sel_p(E/L)^{\Gal(L/K)}$, and the result follows.
\end{proof}

\begin{rem}
We can also deduce Proposition \ref{mainPro} as follows. Let $\mathcal{X'}$ the image of the restriction map $H^1(K,E[p]) \rightarrow H^1(L,E[p])$. Then using inflation-restriction we have a commutative diagram
\begin{small}
$$
\begin{CD}
 0 @>>> H^1(G, E(L)[p]) @>>> H^1(K,E[p]) @>>> \mathcal{X}' @>>> 0 \\
 @.			@VV\psi V							@VV\varphi V				@VV\varphi_L V  @.  \\
 0 @>>> \underset{v\in S}\prod W_{v,L}[p] @>>> \underset{v\in S}\prod H^1(K_v,E)[p] @>>> \underset{v\in S}\prod \hspace{.5mm}\underset{w|v}\prod H^1(L_w,E)[p]	
\end{CD}
$$
\end{small}
and so by the snake lemma we obtain an exact sequence
\begin{equation}\label{p-rank}
0 \longrightarrow \ker\psi \longrightarrow \Sel_p(E/K) \longrightarrow \mathcal{X} \longrightarrow \Big(\prod_{v \in S} W_{v,L}[p]\Big)/\mathrm{Im}\psi \longrightarrow \coker \varphi ,
\end{equation}
where $\mathcal{X}=\ker\varphi_L$. Note that here $\mathcal{X}$ is contained in $\Sel_p(E/L)^G$, and that $\mathcal{X}=\Sel_p(E/L)^G$ if $E(K)[p]=0$.

Suppose now that $E$ has trivial $p$-Selmer group over $K$. Then $\psi$ becomes the zero map and by the Cassels-Poitou-Tate (Proposition \ref{CPT}) exact sequence we have that $\coker \varphi$ is also trivial since it is isomorphic to a subgroup of $\widehat{\Sel_p(E/K)}$, hence (\ref{p-rank}) becomes
$$
\Sel_p(E/L)^G \simeq \prod_{v \in S} W_{v,L}[p].
$$
\end{rem}

Hence in this setting to analyse the $\mathbb{F}_p$-dimension of the $G$-invariants of $p$-Selmer it suffices to know the $\mathbb{F}_p$-dimensions of the local cohomology groups $W_{v,L}$ for $v\in S$. The behaviour of these groups will depend on the reduction type of the places as well as whether or not they lie above $p$. 

\subsection{$\mathbb{F}_p$-dimensions of $W_{v,L}$}

For the remainder of the section assume that $E$ is a semistable elliptic curve defined over $K$. In this section we study the behaviour of $W_{v,L}$ for the different $v\in S$. First note that when $v$ is archimidean or is split in $L$ then $W_{v,L}$ is trivial, hence we only consider inert and ramified primes. These cases are the content of the propositions to follow. Continue to assume that $L/K$ is a finite cyclic extension of number fields of odd prime degree $p$. Let $w$ be a prime of $L$ above $v$. We will write $k_w$ and $k_v$ for the residue fields of $L_w$ and $K_v$, respectively, and we denote their sizes by $q_w$ and $q_v$. Hence $q_w=q_v^{f_v}$ where $f_v$ is the inertial degree of $v$ in $L$. Also, $N_{L_w/K_v}$ will denote the norm map from $L_w$ to $K_v$. First we recall the following well-known result. We let $E_0(L_w)$ denote the subgroup of $E(L_w)$ consisting of points with non-singular reduction and $E_1(L_w)$ denote the kernel of the reduction map $E_0(L_w)\rightarrow \tilde{E}_{\ns}(k_w)$.

\begin{lem}\label{Gseq}
Let $v$ be a place $K$ and $w$ a place of $L$ above $v$, and set $G=\Gal(L_w/K_v)$.  Then there is an exact sequence of $G$-modules
$$
0 \longrightarrow E_1(L_w) \longrightarrow E_0(L_w) \longrightarrow \tilde{E}_{\ns}(k_w) \longrightarrow 0.
$$
\end{lem}

\begin{proof}
The exact sequence follows from \cite{Sil} \S VII.2, Proposition 2.1. Since $E$ has semistable reduction at $v$ it follows that any minimal model for $E$ over $K_v$ stays minimal over $L_w$, hence all the groups in the sequence are $G$-modules. The result follows since the maps in the sequence are $G$-equivariant.
\end{proof}

Recall that if $\widehat E$ is the formal group associated to the elliptic curve $E$, then there is an isomorphism $\widehat E(\mathfrak{m}_{K_v}) \simeq E_1(K_v)$, and in what follows me will sometimes write $\widehat E$ in place of $E_1$. Using this isomoprhism and Proposition \ref{herb} gives that 
$$
H^1(G,E_1(L_w))=H^2(G,E_1(L_w)).
$$
In the next lemma we show the same thing holds for $W_{v,L}$.

\begin{lem}\label{herbW}
Keeping the notation of this section, we have
$$
W_{v,L}\simeq H^2(G,E(L_w)).
$$
\end{lem}

\begin{proof}
The argument is the same as that of Proposition \ref{herb}. Lemma \ref{Gseq} gives the exact sequence
$$
0 \longrightarrow E_1(L_w) \longrightarrow E_0(L_w) \longrightarrow \tilde{E}_{\ns}(k_w) \longrightarrow 0.
$$
By what was just shown $E_1(L_w)$ has trivial Herbrand quotient, as does $\tilde{E}_{\ns}(k_w)$ since it is finite, hence the same is true of $E_0(L_w)$. Finally there is an exact sequence 
$$
0 \longrightarrow E_0(L_w) \longrightarrow E(L_w) \longrightarrow E(L_w)/E_0(L_w) \longrightarrow 0
$$
where $E(L_w)/E_0(L_w)$ is finite, and the result follows by the same argument.
\end{proof}

\subsubsection{Case when $v \nmid p$}

Keeping the notation of this section, let $v$ be a prime in $S$ which does not divide $p$.

\begin{pro}\label{good}
Suppose that $E$ has good reduction at $v$. Then
\begin{itemize}
\item[(i)]
If $v$ is ramified in $L$, then
$$
\dim_{\mathbb{F}_p} W_{v,L} = \dim_{\mathbb{F}_p} \tilde{E}(k_v)[p].
$$
\item[(ii)]
If $v$ is inert in $L$ then $W_{v,L}$ is trivial.
\end{itemize}
\end{pro}

\begin{proof}
Let $w$ be the prime of $L$ above $v$, and set $G=\Gal(L_w/K_v)$, so that $G$ is a cyclic group of order $p$. Recall that $W_{v,L}$ is isomorphic to  $H^1(G,E(L_w))$. By Lemma \ref{Gseq} we have an exact sequence of $G$-modules
$$
0 \longrightarrow E_1(L_w) \longrightarrow E(L_w) \longrightarrow \tilde{E}(k_w) \longrightarrow 0
$$
which induces an exact sequence of cohomology groups
$$
H^1(G,E_1(L_w)) \longrightarrow H^1(G,E(L_w)) \longrightarrow H^1(G,\tilde{E}(k_w)) \longrightarrow H^2(G,E_1(L_w)).
$$
Let $\ell$ denote the characteristic of $k_v$. By \cite{Sil} \S IV, Proposition 2.3 and \S VII, Proposition 2.2 we have that $E_1(L_w)$ is uniquely divisible by $p$, as $\ell\neq p$. It follows then from Proposition \ref{pdiv} that $H^i(G,E_1(L_w))=0$ for $i=1,2$, and therefore we obtain 
$$
H^1(G,E(L_w)) \overset{\sim}{\longrightarrow} H^1(G,\tilde{E}(k_w)).
$$
For (i), as $v$ is totally ramified in $L$, we have $k_w=k_v$ and $G$ acts trivially on it. Hence
$$
 H^1(G,\tilde{E}(k_w)) \simeq \Hom(G,\tilde{E}(k_v))
$$
and the result follows since $G\simeq \mathbb{Z}/p\mathbb{Z}$. 

Part (ii) follows from Theorem 2 of \cite{Lang}, which gives that $H^1(G,\tilde{E}(k_w))$ is trivial, hence so is $W_{v,L}$.
\end{proof}

To deal with primes of split multiplicative reduction we introduce the following notation. Let $w$ be a prime of $L$ above $v$. Then $L_w/K_v$ is tamely ramified with $\Gal(L_w/K_v) \simeq \mathbb{Z}/p\mathbb{Z}$, hence by Lemma 1 of \cite{MJG}, $K_v$ contains a primitive $p$th root of unity so we may write $L_w=K_v(\sqrt[p]{\pi_v})$ where $\pi_v$ is a uniformiser in $K_v$. We then write $u_v$ for a unit in $\mathcal{O}_v$ such that $j(E/K_v)=\pi_v^{-m} u_v$, where $-m=\ord_v(j(E))$.

\begin{pro}\label{split}
Suppose $E$ has split multiplicative reduction at $v$. Then
\begin{itemize}
\item[(i)]
If $v$ is ramified in $L$, then
$$
\dim_{\mathbb{F}_p} W_{v,L}=\begin{cases}
1 & \text{if $u_v^{\frac{q_v-1}{p}} \equiv 1 \mod{\pi_v}$} \\
0 & \text{otherwise}
\end{cases}
$$
\item[(ii)]
If $v$ is inert in $L$ then 
$$
\dim_{\mathbb{F}_p} W_{v,L} = 
\begin{cases}
1 & \text{if $p \mid c_v$} \\
0 & \text{otherwise}
\end{cases}
$$
\end{itemize}
\end{pro}

\begin{proof}
By Lemma \ref{herbW} we have that
$$
W_{v,L} \simeq H^2(G,E(L_w))
$$
and this in turn is simply $E(K_v)/N_{L_w/K_v}E(L_w)$ since $G$ is cyclic.
As $E$ has split multiplicative reduction at $v$, there is a unique $q\in K_v$ with $q \in \mathfrak{m}_{K_v}$ such that $E$ is isomorphic over $K_v$ to the Tate curve $E_q(\overline{K}_v) \simeq \overline{K}_v^\times/q^\mathbb{Z}$. As these isomorphisms are compatible with the action of Galois we obtain
$$
W_{v,L} \simeq \big(K_v^\times/q^\mathbb{Z}\big)\Big/N_{L_w/K_v}\big(L_w^\times/q^\mathbb{Z}\big). 
$$
Consider the commutative diagram
$$
\begin{CD}
1   @>>>   q^\mathbb{Z}   @>>>   L_w^\times   @>>>   L_w^\times/q^\mathbb{Z}  @>>>   1 \\
@.      @VVpV                         @VVN_{L_w/K_v}V      @VVN_{L_w/K_v}V              @.   \\         
1   @>>>   q^\mathbb{Z}   @>>>   K_v^\times   @>>>   K_v^\times/q^\mathbb{Z}  @>>>   1 \\             
\end{CD}
$$
where the rows are exact. From local reciprocity we know that $K_v^\times/N_{L_w/K_v}(L_w^\times)$ is $\mathbb{Z}/p\mathbb{Z}$ hence by the snake lemma we obtain
$$
\dim_{\mathbb{F}_p} W_{v,L} = 
\begin{cases}
1 & \text{if $q\in N_{L_w/K_v}(L_w^\times)$} \\
0 & \text{otherwise}.
\end{cases}
$$

For part (i) let $j=j(E/K_v)$, and recall (\cite{SilAT}, \S V.5) that there is a series $g(T)=T+744T^2+\dots \in\mathbb{Z}[[T]]$ such that
$$
q=g\bigg(\frac{1}{j}\bigg).
$$
Now $j=\pi_v^{-m} u_v$ where $m$ is a positive integer, and so
\begin{align*}
q&=\pi_v^m (u_v^{-1}+\dotsb)\\
  &=\pi_v^m s_v
\end{align*}
where $s_v \equiv u_v^{-1} \mod \pi_v$. Also, since $L_w=K_v(\sqrt[p]{\pi_v})$, it follows that $\pi_v \in N_{L_w/K_v}(L_w^\times)$, hence $q$ is a norm from $L_w$ if and only if $s_v$ is, and the result follows from the computation of the tame Hilbert symbol, see the corollary to Proposition 8 in (\cite{SerreLF}, \S XIV).

For (ii), let $\pi$ be any uniformiser for $K_v$. Then as $v$ is inert, under this choice of uniformiser we may identify $K_v^\times$ with $\mathbb{Z}\times O_v^\times$ and $L_w^\times$ with $\mathbb{Z}\times O_w^\times$, these identifications being compatible with the action of Galois. It follows that under this identification, the field norm maps $(m,u)$ to $(pm,N_{L_w/K_v}(u))$ for $m\in\mathbb{Z}, u\in O_w^\times$. Finally, recall that in unramified extensions the norm map is surjective on the units (see for instance \cite{SerreLF}, \S V.2, Proposition 3) hence we obtain
$$
q\in N_{L_w/K_v}(L_w^\times) \Longleftrightarrow p \mid \ord_v(q).
$$
The result now follows from the fact that $\ord_v(q) = c_v$ (see proof of \cite{SilAT}, \S V.4, Proposition 4.1).
\end{proof}

\begin{pro}\label{nonsplit}
Suppose $E$ has non-split multiplicative reduction at $v$. Then $W_{v,L}$ is trivial.
\end{pro}

\begin{proof}
The exact sequence
$$
0 \longrightarrow E_1(L_w) \longrightarrow E_0(L_w) \longrightarrow \tilde{E}_{\ns}(k_w) \longrightarrow 0
$$
induces an exact sequence of cohomology groups
$$
H^1(G,E_1(L_w)) \longrightarrow H^1(G,E_0(L_w)) \longrightarrow H^1(G,\tilde{E}_\mathrm{ns}(k_w)) \longrightarrow H^2(G,E_1(L_w))
$$
and by the same argument as in Proposition \ref{good} we obtain
$$
H^1(G,E_0(L_w)) \simeq  H^1(G,\tilde{E}_\mathrm{ns}(k_w)).
$$

First suppose $v$ is inert. Then again using Theorem 2 of \cite{Lang} we obtain that $H^1(G,E_0(L_w))$ is trivial. Also, we have an exact sequence 
$$
0 \longrightarrow E_0(L_w) \longrightarrow E(L_w) \longrightarrow E(L_w)/E_0(L_w) \longrightarrow 0
$$
which induces the exact sequence
$$
H^1(G,E_0(L_w)) \longrightarrow H^1(G,E(L_w)) \longrightarrow H^1(G, E(L_w)/E_0(L_w)).
$$
As $E$ has non-split multplicative reduction at $v$, it follows that the group $E(L_w)/E_0(L_w)$ is of order at most $2$, hence all the terms in the above sequence are trivial.

Now suppose $v$ is ramified in $L$. Recall $q_v$ denotes the size of the residue field $k_v$, where $q_v$ is coprime to $p$. Then as $v$ is ramified, it follows that $p$ divides $q_v-1$. However $\tilde{E}_\mathrm{ns}(k_w)$ is a group if size $q_v+1$, so $H^1(G,\tilde{E}_\mathrm{ns}(k_w))$ must be trivial, hence by the same argument as above so is $W_{v,L}$.

\end{proof}


\subsubsection{Case when $v\mid p$}


We now wish to study the groups $W_{v,L}$ for primes $v$ which divide $p$. In this case the formal group $E_1(L_w)$ is not uniquely divisible by $p$, so its cohomology groups are not necessarily trivial. The situation depends on the reduction type of $E$ at $v$, and results from section 2 will be useful particularly when $E$ has good supersingular reduction at $v$. The case of ordinary reduction was studied by Mazur and Rubin in \cite{MazurRubin}. We say that $E$ has \emph{anomalous reduction} at $v$ if $\tilde{E}(k_v)[p]$ is non-trivial.

\begin{pro}\label{Ordinary}
Let $E$ have good ordinary reduction at $v$. Then

\begin{itemize}
\item[(i)] If $v$ is ramified in $L$, then
$$
\dim_{\mathbb{F}_p} W_{v,L}=
\begin{cases}
1\text{ or } 2 & \text{if $E$ has anomalous reduction at $v$} \\
0 & \text{otherwise}.
\end{cases}
$$

\item[(ii)] If $v$ is inert in $L$, then $W_{v,L}$ is trivial.
\end{itemize}

\end{pro}

\begin{proof}
Part (i) follows immediately from Proposition B.2 in \cite{MazurRubin}. For part (ii), note that by the same argument as above we have an exact sequence of cohomology groups
$$
H^1(G,E_1(L_w)) \longrightarrow H^1(G,E(L_w)) \longrightarrow H^1(G,\tilde{E}(k_w)) \longrightarrow H^2(G,E_1(L_w)).
$$
where $w$ is a prime of $L$ above $v$. Denote by $\widehat E$ the formal group associated to the elliptic curve $E$. Now since $v$ is unramified in $L$, the norm map on $\widehat E$ is surjective and hence its cohomology is trivial. Finally, again using Theorem 2 of \cite{Lang} we have that $H^1(G,\tilde{E}(k_w))$ is trivial and therefore the same thing is true of $W_{v,L}$.

\end{proof}

\begin{pro}\label{Supersingular}
Suppose $E$ has good supersingular reduction at $v$.
\begin{itemize}
\item[(i)] Suppose $v$ is ramified in $L$. Let $w$ be the place of $L$ above $v$, and let $\pi_L$ be a uniformizer for the completion $L_w$. Set $t=\ord_w(\sigma(\pi_L)-\pi_L)-1$. Then 
$$
\dim_{\mathbb{F}_p} W_{v,L}=0 \Longleftrightarrow t=1
$$

Moreover, for $t\geqslant 2$ we have 
$$
f_K \leqslant \dim_{\mathbb{F}_p} W_{v,L} \leqslant \min\{f_K(t-1), [K:\mathbb{Q}_p]+2\}.
$$

\item[(ii)] If $v$ is inert in $L$, then $W_{v,L}$ is trivial.
\end{itemize}
\end{pro}

\begin{proof}
Consider again the exact sequence of cohomology groups
$$
E(K_v) \longrightarrow \tilde{E}(k_v) \longrightarrow H^1(G,E_1(L_w)) \longrightarrow H^1(G,E(L_w)) \longrightarrow H^1(G,\tilde{E}(k_w)).
$$
As $E$ has supersingular reduction at $v$ we have $p \nmid |E(k_w)|$ and so $H^1(G,\tilde{E}(k_w))=0$. Also, the map $E(K_v) \rightarrow \tilde{E}(k_v)$ is simply the reduction map, hence it is surjective. These together imply that we have an isomorphism 
$$
H^1(G,E_1(L_w)) \simeq H^1(G,E(L_w))=W_{v,L}.
$$
For part (i) $\widehat{E}$ has height $2$ as $E$ has supersingular reduction at $v$, hence the result follows from propositions \ref{SurjectiveNorm} and \ref{herb}. Part (ii) follows from the same reasoning as Proposition \ref{Ordinary}.
\end{proof}

\begin{proof}[Proof of Theorem \ref{MainThm2}]
By Proposition \ref{mainPro},
\begin{equation*}
\dim_{\mathbb{F}_p} \Sel_p(E/L)^{\Gal(L/K)} = \sum_{v \in S} \dim_{\mathbb{F}_p} W_{v,L}.
\end{equation*}
If we let $\delta_v = \dim_{\mathbb{F}_p} W_{v,L}$ then the theorem then follows from propositions \ref{good}, \ref{split}, \ref{nonsplit}, \ref{Ordinary} and \ref{Supersingular} which compute the $\mathbb{F}_p$-dimensions of $W_{v,L}$ for the different primes in $S$.
\end{proof}

\begin{proof}[Proof of Corollary \ref{MainCor}]
 By Lemma 1 in (\cite{SerreLF}, \S IX) we have that $\Sel_p(E/L)$ is trivial if and only if $\Sel_p(E/L)^{\Gal(L/K)}$ is trivial, and the result follows.
\end{proof}

\section{Extensions of the form $K(\sqrt[p]{m})$}

In this section we will finish the proof of Theorem \ref{MainThm1}. Let then $E$ be an elliptic curve over $\mathbb{Q}$. Consider the extension of $E$ to $K=\mathbb{Q}(\zeta_p)$ and let $L_m=K(\sqrt[p]{m})$. For a fixed $m$, if $v$ is a prime of $K$ then again we let $\delta_v=\dim_{\mathbb{F}_p} W_{v,L_m}$. We also denote by $\ell_v$ the prime of $\mathbb{Q}$ below $v$, and set $q_v$ to be the size of the residue field of $K_v$. That is, if $f_v$ is the order of $\ell_v \pmod{p}$, then $q_v=\ell_v^{f_v}$.

Note that most of the statement follows directly from Theorem \ref{MainThm2},  and the only statements which need to be modified are those for primes $v \nmid p$ of split multiplicative reduction ramified in $L_m$, and primes $v|p$ of good supersingular reduction ramified in $L_m$. This is done in the two propositions below.  To deal with the split multiplicative reduction case recall the notation introduced in section 1. Let $j$ denote the $j$-invariant of $E$ and let $v \nmid p$ be a prime of $K$ at which $E$ has split multiplicative reduction. Note that since $\ell_v$ is unramified in $K$, it follows that $E$ has semistable reduction at $\ell_v$. For each positive integer $m$ let $s\geqslant 0$ be such that $m=\ell^s d$ where $d$ is coprime to $\ell$. Since $E$ is semistable at $\ell_v$ we may write $j=\ell_v^{-n}\frac{a}{b}$ with $n$ a positive integer and $a$ and $b$ coprime to $\ell_v$. Then we set
\begin{equation}\label{ellunit2}
u_{\ell_v,m} :=d^n\bigg (\frac{a}{b} \bigg)^s.
\end{equation}

\begin{pro}\label{split2}
Let $v$ be a prime of $K$ such that $E$ has split multiplicative reduction at $v$. Then
$$
\delta_v = \begin{cases}
1 & \text{if $u_{\ell_v,m}^{\frac{q_v-1}{p}} \equiv 1 \mod \ell_v$}\\
0 & \text{otherwise}
\end{cases}
$$
\end{pro}

\begin{proof}
We have that $\delta_v = 1$ if and only if the tame Hilbert symbol $(m,q)_v$ is equal to 1. We argue as in the proof of part (i) of Proposition \ref{split}. We have that
\begin{align*}
q &= g\bigg(\frac{1}{j}\bigg) \\
   &=  g\bigg(\ell_v^n \frac{b}{a}\bigg)\\
   &= \ell_v^n s_v
\end{align*}
where $s_v \equiv \frac{b}{a} \pmod{\ell_v}$. Now again a calculation using the corollary to Proposition 8 in (\cite{SerreLF}, \S XIV) gives the desired result.
\end{proof}

Now consider $v$ a prime of $K$ dividing $p$, and again we let $\delta_v=\dim_{\mathbb{F}_p} W_{v,L_m}$ which as we have seen is equal to the $\mathbb{F}_p$-dimension of $H^1(G,E_1(L_w))$, where $E_1$ can be identified with the formal group of $E$. The only thing left to show is that when considering extensions of the form $K(\sqrt[p]{m})$, then $\delta_v=p-2$ if $p \mid m$, and equals $0$ otherwise. This is shown in the following proposition, which is done for general formal groups of height $h>1$.

\begin{pro}\label{Kmextensions}
Let $\mathcal{F}$ be a formal group over $\mathbb{Q}_p$ of height $h>1$. Let $K=\mathbb{Q}_p(\zeta_p)$ and $L_m=K(\sqrt[p]{m})$. Then
$$
\dim_{\mathbb{F}_p} \mathcal{F}(\mathfrak{m}_K)/N^{\mathcal{F}}_{L_m/K}(\mathcal{F}(\mathfrak{m}_{L_m}))=
\begin{cases}
p-2 & \text{if $p\mid m$} \\
0 & \text{otherwise}
\end{cases}
$$
\end{pro}

\begin{proof}
Let $t_m$ be the largest for which the higher ramification group of $L_m/K$ is non-trivial. Then one can show that $t_m=p-1$ when $p \mid m$ and $t_m \leqslant 1$ otherwise. The rest of the proof is a very slight modification of that of Proposition \ref{SurjectiveNorm}. Indeed, when $p \mid m$ then $t_m=p-1$ and so it follows by Proposition \ref{SurjectiveNorm} that $\mathcal{F}(\mathfrak{m}_{L_m}^{p-1})$ surjects via the norm onto $\mathcal{F}(\mathfrak{m}_K^{p-1})$. Recall from Corollary \ref{NormCor} that for $x\in\mathcal{F}(\mathfrak{m}_{L_m})$ we have
$$
N^{\mathcal{F}}_{L_m/K}(x) \equiv \Tr_{L_m/K}(x)+\sum_{i=1}^\infty a_i N_{L/K}(x)^i \mod \Tr_{L_m/K}(x^2\mathcal{O}_{L_m}).
$$
Also, since $h>1$ then $v_K(a_i)\geqslant v_K(p)=p-1$ unless $i=kp^{h-1}$, since $\mathcal{F}$ is defined over $\mathbb{Q}_p$. But by Lemma \ref{TraceLemma} we have that $\Tr_{L_m/K}(\mathfrak{m}_{L_m})=\mathfrak{m}_K^{p-1}$ and so it follows that $N^{\mathcal{F}}_{L_m/K}(\mathcal{F}(\mathfrak{m}_{L_m})) \subset \mathcal{F}(\mathfrak{m}_K^{p-1})$. As we have already seen that the reverse inclusion holds, we conclude that 
$$
N^{\mathcal{F}}_{L_m/K}(\mathcal{F}(\mathfrak{m}_{L_m})) = \mathcal{F}(\mathfrak{m}_K^{p-1}),
$$
as desired. Finally, when $p$ does not divide $m$ we have that $t_m \leqslant 1$ and by Proposition \ref{SurjectiveNorm} we have that 
$$
N^{\mathcal{F}}_{L_m/K}(\mathcal{F}(\mathfrak{m}_{L_m})) = \mathcal{F}(\mathfrak{m}_K).
$$
This concludes the proof.
\end{proof}

\begin{proof}[Proof of Theorem \ref{MainThm1}]
This follows immediately from Theorem \ref{MainThm2} as well as propositions \ref{split2} and \ref{Kmextensions}.
\end{proof}

\nocite{Sil}
\nocite{Lang}
\nocite{SerreLF}
\nocite{GCEC}
\nocite{Gr}
\nocite{MatsSha}
\nocite{Haze}
\nocite{Maz}
\nocite{CG}
\nocite{TD}
\nocite{DDIwa}
\nocite{HM}
\nocite{MazurRubinHilbert}
\nocite{Newton}
\nocite{CasselsSelmer}
\nocite{Fisher}
\nocite{Kloost}
\nocite{KloostSchaef}
\nocite{Kramer}
\nocite{Matsuno2}
\nocite{Kramer2}
\nocite{Bartel}
\nocite{Bartel2}
\nocite{Cesna}
\nocite{magma}

\bibliographystyle{alpha}
\bibliography{Selm.bib}{}

\end{document}